\documentclass{amsart}
\usepackage{mathptmx, amscd, amssymb, amsmath, enumerate, slashbox, hyperref, float}

\DeclareMathAlphabet{\mathcal}{OMS}{cmsy}{m}{n}

\newtheorem{theorem}{Theorem}[section]
\newtheorem{lemma}[theorem]{Lemma}
\newtheorem{corollary}[theorem]{Corollary}
\newtheorem{proposition}[theorem]{Proposition}

\theoremstyle{definition}
\newtheorem{definition}[theorem]{Definition}
\newtheorem{example}[theorem]{Example}
\newtheorem{remark}[theorem]{Remark}

\newtheorem{question}[theorem]{Question}
\numberwithin{equation}{theorem}

\def\ge{\geqslant}
\def\le{\leqslant}
\def\phi{\varphi}
\def\tilde{\widetilde}
\def\bar{\overline}
\def\to{\longrightarrow}

\def\mapsto{\longmapsto}
\def\into{\lhook\joinrel\longrightarrow}
\def\onto{\relbar\joinrel\twoheadrightarrow}

\def\dlim{\varinjlim}
\def\ilim{\varprojlim}
\def\codim{\operatorname{codim}}
\def\height{\operatorname{height}\,}
\def\rank{\operatorname{rank}}
\def\Hom{\operatorname{Hom}}
\def\Hsing{H_{\mathrm{sing}}}
\def\Ext{\operatorname{Ext}}
\def\Tor{\operatorname{Tor}}
\def\pd{\operatorname{pd}}
\def\Pic{\operatorname{Pic}}
\def\Proj{\operatorname{Proj}}
\def\Sing{\operatorname{Sing}}
\def\Spec{\operatorname{Spec}\,}
\def\Supp{\operatorname{Supp}}

\newcommand{\Sym}{\mathrm{Sym}}
\newcommand{\R}{\mathrm{R}}
\newcommand{\dR}{\mathrm{dR}}

\def\frakm{\mathfrak{m}}

\def\frakp{\mathfrak{p}}

\def\CC{\mathbb{C}}

\def\FF{\mathbb{F}}

\def\PP{\mathbb{P}}
\def\QQ{\mathbb{Q}}

\def\ZZ{\mathbb{Z}}

\def\calA{\mathcal{A}}
\def\calC{\mathcal{C}}

\def\calE{\mathcal{E}}
\def\calF{\mathcal{F}}
\def\calG{\mathcal{G}}
\def\calH{\mathcal{H}}
\def\calI{\mathcal{I}}
\def\calJ{\mathcal{J}}
\def\calL{\mathcal{L}}
\def\calM{\mathcal{M}}
\def\calN{\mathcal{N}}
\def\calO{\mathcal{O}}
\def\calQ{\mathcal{Q}}

\def\calT{\mathcal{T}}

\def\calHom{\mathcal{H}om}

\begin{document}
\title{Stabilization of the cohomology of thickenings}

\author[Bhatt]{Bhargav Bhatt}
\address{Department of Mathematics, University of Michigan, 530 Church Street, Ann Arbor, MI~48109, USA}
\email{bhargav.bhatt@gmail.com}

\author[Blickle]{Manuel Blickle}
\address{Institut f\"ur Mathematik, Fachbereich 08, Johannes Gutenberg-Universit\"at Mainz,
\newline 55099~Mainz, Germany}
\email{blicklem@uni-mainz.de}

\author[Lyubeznik]{Gennady Lyubeznik}
\address{Department of Mathematics, University of Minnesota, 206 Church~St., Minneapolis,\newline MN~55455, USA}
\email{gennady@math.umn.edu}

\author[Singh]{Anurag K. Singh}
\address{Department of Mathematics, University of Utah, 155 S 1400 E, Salt Lake City,\newline UT~84112, USA}
\email{singh@math.utah.edu}

\author[Zhang]{Wenliang Zhang}
\address{Department of Mathematics, Statistics, and Computer Science, University of Illinois at Chicago, 851 S.~Morgan~St., Chicago, IL 60607, USA}
\email{wlzhang@uic.edu}

\thanks{B.B.~was supported by NSF grants DMS~1501461 and DMS~1522828, and by a Packard Fellowship, M.B.~by DFG grant SFB/TRR45, G.L.~by NSF grant DMS~1161783, A.K.S.~by NSF grants DMS~1162585 and DMS~1500613, and W.Z.~by NSF grant DMS~1405602. The authors are grateful to the American Institute of Mathematics (AIM) for supporting their collaboration.}

\begin{abstract}
For a local complete intersection subvariety $X=V(\calI)$ in $\PP^n$ over a field of characteristic zero, we show that, in cohomological degrees smaller than the codimension of the singular locus of $X$, the cohomology of vector bundles on the formal completion of~$\PP^n$ along $X$ can be effectively computed as the cohomology on any sufficiently high thickening~$X_t=V(\calI^t)$; the main ingredient here is a positivity result for the normal bundle of~$X$. Furthermore, we show that the Kodaira vanishing theorem holds for all thickenings~$X_t$ in the same range of cohomological degrees; this extends the known version of Kodaira vanishing on $X$, and the main new ingredient is a version of the Kodaira-Akizuki-Nakano vanishing theorem for $X$, formulated in terms of the cotangent complex.
\end{abstract}
\maketitle

%%%%%%%%%%%%%%%%%%%%%%%%%%%%%%%%%%%%%%%%%%%%%%%%%%%
\section{Introduction}
\label{section:introduction}
%%%%%%%%%%%%%%%%%%%%%%%%%%%%%%%%%%%%%%%%%%%%%%%%%%%

Let $X$ be a closed subscheme of projective space $\PP^n$ over a field, defined by a sheaf of ideals $\calI$. For each integer $t \ge 1$, let $X_t \subset \PP^n$ be the $t$-th thickening of $X$, i.e., the closed subscheme defined by the sheaf of ideals $\calI^t$. As $t$ varies, we obtain an inductive system
\[
\cdots \to X_{t-1} \to X_t \to X_{t+1} \to \cdots
\]
of subschemes of projective space. The goal of this paper is to investigate the behavior, especially the stability properties, of coherent sheaf cohomology in this system. More precisely, for each cohomological degree $k$ and each locally free coherent sheaf $\calF$ on $\PP^n$, we have an induced projective system of finite rank vector spaces
\begin{equation}
\label{equation:projective:system}
\cdots \to H^k(X_{t+1},\calF_{t+1}) \to H^k(X_t,\calF_t) \to H^k(X_{t-1},\calF_{t-1}) \to \cdots,
\end{equation}
where $\calF_t=\calO_{X_t}\otimes_{\calO_{\PP^n}}\calF$ is the induced sheaf on $X_t$. If $X$ is smooth, and the characteristic of the ground field is zero, Hartshorne proved that the inverse limit of this sequence is also a finite rank vector space for $k<\dim X$, \cite[Theorem~8.1~(b)]{Hartshorne:ampleVB}. The proof proceeds via showing that the transition maps in the system are eventually isomorphisms. 

If $X$ is locally a complete intersection (lci), but not necessarily smooth, the inverse limit of the system~\eqref{equation:projective:system} is still finite rank for $k < \dim X$, see Corollary~\ref{corollary:finite:rank}. Yet, in the lci case, the maps in the system are {\em not} eventually isomorphisms; in particular, Hartshorne's proof that the inverse limit has finite rank cannot be extended to the lci case. 

This paper is the result of an attempt to understand when the maps in~\eqref{equation:projective:system} are, in fact, eventually isomorphisms. One of our main theorems gives a sharp answer to this:

\begin{theorem}
\label{theorem:constant}
Let $X$ be a closed lci subvariety of $\PP^n$ over a field of characteristic zero. Let~$\calF$ be a coherent sheaf on $\PP^n$ that is flat along $X$. Then, for each $k<\codim(\Sing X)$, the natural map
\[
H^k(X_{t+1},\calO_{X_{t+1}}\otimes_{\calO_{\PP^n}}\calF) \to H^k(X_t,\calO_{X_t}\otimes_{\calO_{\PP^n}}\calF)
\]
is an isomorphism for all $t\gg0$.

In particular, if $X$ is smooth or has at most isolated singular points, then the map above is an isomorphism for $k<\dim X$ and $t\gg0$.
\end{theorem}

While Hartshorne's result was ``one of the principal motivations for developing the theory of ample vector bundles" \cite[\S0]{Hartshorne:ampleVB}, our proof of the above theorem uses the notion of partially ample vector bundles introduced in \cite{Arapura:AJM}.

The theorem is sharp in multiple ways. First, the assertion fails in positive characteristic, even for $X$ smooth, by Example~\ref{example:charp}. Secondly, the bounds on the cohomological degree in terms of the dimension of the singular locus are sharp as well: this is essentially trivial in the smooth case, and is illustrated by Example~\ref{example:slocus:sharp} in the general case. Finally, the stabilization fails in the absence of the lci hypothesis; see Example~\ref{example:singular}.

Theorem~\ref{theorem:constant} translates to local cohomology as follows: Let $R$ be a standard graded polynomial ring that is finitely generated over a field~$R_0$ of characteristic zero; set $\frakm$ to be the homogeneous maximal ideal of $R$. For a homogeneous prime ideal $I$, Theorem~\ref{theorem:constant} yields stability results for graded components of the modules $H^k_\frakm(R/I^t)$, see Proposition~\ref{proposition:equivalence}.

For a noetherian ring $R$, the local cohomology module $H^k_I(R)$ can be realized as a direct limit of the modules $\Ext^k(R/I^t,R)$. In \cite{EMS}, Eisenbud, Musta{\c{t}}{\u{a}}, and Stillman ask about the injectivity of the maps in the inductive system; this is discussed in~\S\ref{section:EMS}. Theorem~\ref{theorem:constant} translates to a partial answer to their question in the graded context, after restricting to an arbitrary but fixed graded component:

\begin{corollary}
\label{corollary:ext}
Let $R$ be a standard graded polynomial ring that is finitely generated over a field~$R_0$ of characteristic zero. Let $I$ be a homogeneous prime ideal such that $X=\Proj R/I$ is smooth. Fix an integer $j$ and a nonnegative integer $k$. Then the natural map
\[
{\Ext^k_R(R/I^t,R)}_j \to {H^k_I(R)}_j
\]
is injective for $t\gg0$.

When $k=\height I$, the displayed map is injective for each $t\ge1$, whereas if~$k>\height I$, then it is an isomorphism for $t\gg0$. (The modules vanish for $k<\height I$.)

More generally, suppose that~$X$ is lci. Then, for each $k>\height I + \dim(\Sing X)$, the displayed map is an isomorphism for all~$t\gg0$.
\end{corollary}

Corollary~\ref{corollary:ext} is sharp in the sense that the least $t$ for which injectivity holds in the display depends on $j$, see Example~\ref{example:nonmax}. The bounds on the cohomological degree in terms of the dimension of the singular locus are sharp as well, as with Theorem~\ref{theorem:constant}.

The proof of Theorem~\ref{theorem:constant} occupies~\S\ref{section:constant}; we first record the case of $X$ smooth, so as to set the stage for what follows. Indeed, in the smooth case, the proof is a straightforward consequence of the ampleness of the normal bundle; to extend this to the singular case, we rely on the theory of partially ample vector bundles from~\cite{Arapura:AJM}. Our main result here says, roughly, that sections of symmetric powers of the normal bundle can move freely in $\codim(\Sing X)$ many directions. More precisely, we have:

\begin{theorem}
Let $X$ be a closed lci subvariety of $\PP^n$, with singular locus of dimension~$s$. Then the normal bundle $\calN$ is $s$-ample, i.e., for each coherent sheaf $\calF$ on $X$, we have~$H^i(X, \Sym^t(\calN) \otimes \calF) = 0$ for $i > s$ and $t \gg 0$.

In particular, if $X$ has at most isolated singular points, then $\calN$ is ample.
\end{theorem}

In \S\ref{section:kodaira}, we tackle the Kodaira vanishing theorem for thickenings. Recall that for $X$ smooth of characteristic zero, the Kodaira vanishing theorem gives $H^k(X,\calO_X(-j))=0$ for integers~$k$ and $j$ with $k<\dim X$ and $j>0$; more generally, there exists a similar statement for~$k < \codim(\Sing X)$ for Cohen-Macaulay $X$, see \cite{Arapura:PAMS}. In~\S\ref{section:kodaira}, we extend this result to the following version of the Kodaira vanishing theorem for thickenings:

\begin{theorem}
\label{theorem:kodaira}
Let $X$ be a closed lci subvariety of $\PP^n$ over a field of characteristic zero. Then, for each $k$ with $k < \codim(\Sing X)$, and $j > 0$, we have
\[
H^k(X_t, \calO_{X_t}(-j)) = 0.
\]
In particular, if $X$ is smooth or has at most isolated singular points, then the vanishing above holds for each $k<\dim X$ and $j>0$.
\end{theorem}

This result was motivated by the question of \emph{effectivity} in Theorem~\ref{theorem:constant}: for $\calF = \calO_{\PP^n}(j)$ with $j < 0$, Theorem~\ref{theorem:kodaira} implies that the map
\[
H^k(X_{t+1},\calO_{X_{t+1}}(j)) \to H^k(X_t,\calO_{X_t}(j))
\]
in Theorem~\ref{theorem:constant} is an isomorphism for \emph{all} $t$, as the groups vanish. The same conclusion obviously cannot hold for $j > 0$. Nevertheless, for all $j \in \ZZ$ and $k < \codim(\Sing X) - 1$, we have \emph{effectivity} in the following sense: the map displayed above is an isomorphism for integers $t \ge t_0(j)$, where
\[
t_0(j) := \dim X + j + 2
\]
(see Remark~\ref{remark:effective:constancy} for more). In particular, the bound $\dim X + j + 2$ suffices for any embedding of the underlying variety $X$.

The main tool used to prove Theorem~\ref{theorem:kodaira} for $X$ smooth is the Kodaira-Akizuki-Nakano (KAN) vanishing theorem; in fact, the two results are essentially equivalent. Thus, to prove Theorem~\ref{theorem:kodaira} in the singular case, we need a version of the KAN theorem for singular spaces. To the best of our knowledge, known results in this direction do not suffice for this purpose; see \S\ref{section:kodaira} for further discussion. Nevertheless, we prove the following version of the KAN theorem, formulated in terms of the cotangent complex $L_X$ and its \emph{derived} exterior powers, which suffices for the application in the proof of Theorem~\ref{theorem:kodaira}:

\begin{theorem}
\label{theorem:KAN:vanishing}
Let $X$ be a closed lci subvariety of $\PP^n$ over a field of characteristic zero.~Then
\[
H^a(X, \wedge^b L_X(-j)) = 0
\]
for $a+b < \codim(\Sing X)$ and $j > 0$. 
\end{theorem}

The statement is also valid for $X$ non-reduced (and still carries some content), and is fairly optimal: the lci assumption is necessary, and the ``Serre dual'' formulation of the smooth case, which says
\[
H^a(X, \wedge^b L_X(j)) = 0
\]
for $a + b > \dim X$ and $j > 0$, fails dramatically in the singular case; see \S\ref{section:kodaira} for a further discussion of these issues, with examples.

\begin{remark}
In the situation of Theorem~\ref{theorem:KAN:vanishing}, results of Lebelt \cite{Lebelt75,Lebelt77}, along with the Auslander-Bridger syzygy theorem \cite[Theorem~3.8]{Evans-Griffith}, imply that $\wedge^b L_X$ coincides with $\Omega^b$ for small $b$ (roughly $b < d$). Thus, in this range, the above vanishing theorem is a statement about a coherent sheaf. For large $b$, this need not be the case, \emph{cf.}~\cite[Lemma~2.1]{Bhatt:torsion}. We thank Lucho Avramov for bringing Lebelt's work to our attention.
\end{remark}

%%%%%%%%%%%%%%%%%%%%%%%%%%%%%%%%%%%%%%%%%%%%%%%%%%%
\section{Stabilization of the cohomology of thickenings}
\label{section:constant}
%%%%%%%%%%%%%%%%%%%%%%%%%%%%%%%%%%%%%%%%%%%%%%%%%%%

We begin by recording some equivalent formulations of the assertion of Theorem~\ref{theorem:constant}; in particular, the following proposition shows that Corollary~\ref{corollary:ext} follows from Theorem~\ref{theorem:constant}.

\begin{proposition}
\label{proposition:equivalence}
Let $R=\FF[x_0,\dots,x_n]$ be a standard graded polynomial ring over a field~$\FF$, and let $I$ be a homogeneous ideal. Set $X=\Proj R/I$ and fix an integer $\ell$.

Then the following are equivalent:
\begin{enumerate}[\quad\rm(1)]
\item For each integer $k$ with $k<\ell$, and each $j$, the maps
\[
H^k(X_{t+1},\calO_{X_{t+1}}(j)) \to H^k(X_t,\calO_{X_t}(j))
\]
are isomorphisms for all $t\gg0$.

\item For each integer $k$ with $k<\ell$, and each degree $j$, the maps
\[
{H^{k+1}_\frakm(R/I^{t+1})}_j \to {H^{k+1}_\frakm(R/I^t)}_j
\]
are isomorphisms for all $t\gg0$, where $\frakm$ is the homogeneous maximal ideal of $R$.

\item For each integer $k$ with $k<\ell$, and each degree $j$, the maps
\[
{\Ext^{n-k}_R(R/I^t,R)}_j \to {\Ext^{n-k}_R(R/I^{t+1},R)}_j
\]
are isomorphisms for all $t\gg0$.
\end{enumerate}
\end{proposition}

\begin{proof}
For the equivalence of (1) and (2), note that for $k\ge 1$, one has
\[
H^k(X_t,\calO_{X_t}(j)) = {H^{k+1}_\frakm(R/I^t)}_j,
\]
while the remaining case follows from the exact sequence
\begin{equation}
\label{equation:four:sequence}
0 \to {H^0_\frakm(R/I^t)}_j \to {(R/I^t)}_j \to H^0(X_t,\calO_{X_t}(j)) \to {H^1_\frakm(R/I^t)}_j \to 0,
\end{equation}
since, for each $j$, one has $(R/I^t)_j=(R/I^{t+1})_j$ for $t\gg0$.

The equivalence of (2) and (3) is immediate from graded local duality, see for example \cite[Theorem~3.16.19]{Bruns-Herzog}: applying the dualizing functor $(-)^\vee$ to
\[
{H^{k+1}_\frakm(R/I^{t+1})}_j \to {H^{k+1}_\frakm(R/I^t)}_j
\]
yields the map
\[
{\Ext^{n-k}_R(R/I^t, R)}_{-n-1-j} \to {\Ext^{n-k}_R(R/I^{t+1}, R)}_{-n-1-j}.\qquad\qedhere
\]
\end{proof}

We next record a finiteness result for the cohomology of the formal completion of $\PP^n$ along a closed subvariety $X$:

\begin{theorem}
\label{theorem:formal:finite}
Let $X$ be a closed subscheme of $\PP^n$ over a field $\FF$. Let~$\calF$ be a coherent sheaf on $\PP^n$ that is flat along $X$. Set $\calF_t=\calO_{X_t}\otimes_{\calO_{\PP^n}}\calF$ to be the induced sheaf on $X_t$. Suppose~$d$ is a positive integer such that the local cohomology sheaves $\calH^i_X(\PP^n, \calO_{\PP^n})$ of $\calO_{\PP^n}$ with support in $X$ vanish for all $i>n-d$. Then, for each integer $k$ with $k<d$, the inverse limit
\[
\ilim_t H^k(X_t,\calF_t)
\]
is a finite rank $\FF$-vector space. If $\calF=\calO_{\PP^n}(j)$ for $j<0$, this inverse limit vanishes.
\end{theorem}

\begin{proof}
Let $R=\FF[x_0,\dots,x_n]$, and let $I$ be the homogeneous ideal of $R$ defining $X$. We first consider the case $\calF=\calO_{\PP^n}(j)$. Fix an integer $i$ with $i>n-d$. The local cohomology module $H^i_I(R)$ is a graded $R$-module, with $\calH^i_X(\PP^n, \calO_{\PP^n})$ the associated quasi-coherent sheaf on $\PP^n$. Since $\calH^{i}_X(\PP^n, \calO_{\PP^n})$ vanishes, the support of $H^i_I(R)$ must be a subset of~$\{\frakm\}$. 

The injective dimension of $H^i_I(R)$ as an $R$-module is bounded by its dimension: this is \cite[Corollary~3.6]{Lyubeznik:Invent} in the case that $\FF$ has characteristic zero, and \cite[Corollary~3.7]{HS:TAMS} in the case of positive characteristic; a characteristic-free proof is provided in \cite{Lyubeznik:injective}. Thus,~$H^i_I(R)$ is injective with $\Supp H^i_I(R)\subseteq\{\frakm\}$. It follows that~$H^i_I(R)$ is a direct sum of copies of~$H^{n+1}_\frakm(R)$; moreover, it is a finite direct sum by the finiteness of the Bass numbers of~$H^i_I(R)$, proved in the same references. The isomorphism
\begin{equation}
\label{equation:bass}
H^i_I(R) \cong H^{n+1}_\frakm(R)^{\oplus\mu}
\end{equation}
is degree-preserving by \cite[Theorem~1.1]{Ma-Zhang} and \cite[Theorem~3.4]{YiZhang}. Since each graded component of $H^{n+1}_\frakm(R)$ has finite rank, the same is true of $H^i_I(R)$. Note that the graded canonical module of $R$ is $\omega_R=R(-n-1)$. Let $j$ be an arbitrary integer. Applying the dualizing functor to the inductive system
\[
\cdots\to {\Ext^i_R(R/I, \omega_R)}_{-j} \to {\Ext^i_R(R/I^2, \omega_R)}_{-j} \to {\Ext^i_R(R/I^3, \omega_R)}_{-j} \to\cdots,
\]
which has limit ${H^i_I(\omega_R)}_{-j} = {H^i_I(R)}_{-j-n-1}$, shows that the inverse limit
\[
\ilim_t {H^{n+1-i}_\frakm(R/I^t)}_j = {(H^i_I(R)^\vee)}_{j+n+1}
\]
has finite rank. Set $k=n-i$ above; the condition $i>n-d$ translates as $k<d$. For~$k\ge 1$, it follows that
\[
\ilim_t H^k(X_t,\calO_{X_t}(j)) = \ilim_t H^{k+1}_\frakm(R/I^t)_j
\]
has finite rank. The remaining case, $k=0$, follows from the exact sequence~\eqref{equation:four:sequence}, since a projective system of finite rank vector spaces satisfies the Mittag-Leffler condition.

Lastly, suppose $j<0$. Then, for each $k<d$, one has
\[
\ilim_t H^k(X_t,\calO_{X_t}(j)) = \ilim_t H^{k+1}_\frakm(R/I^t)_j = ({H^{n-k}_I(R)}_{-j-n-1})^\vee
\]
but this vanishes since
\[
{H^{n-k}_I(R)}_{-j-n-1} \cong {H^{n+1}_\frakm(R)^\mu}_{-j-n-1}
\]
by~\eqref{equation:bass}, and $H^{n+1}_\frakm(R)$ vanishes in degrees greater that $-n-1$. This completes the proof in the case that $\calF=\calO_{\PP^n}(j)$.

We now handle the general case where $\calF$ is an arbitrary coherent sheaf on $\PP^n$ that is a vector bundle on a Zariski open neighborhood $U$ of $X$ in $\PP^n$. Let $\calF^\vee := \calHom_{\calO_{\PP^n}}(\calF,\calO_{\PP^n})$ be the dual of $\calF$, and note that $\calF^\vee$ is also a coherent sheaf that is a vector bundle over~$U$. Let $G^\bullet \to \calF^\vee$ be a finite resolution of $\calF^\vee$ with each $G^i$ being a finite direct sum of copies of twists of $\calO_{\PP^n}$, and $G^i = 0$ for $i > 0$ or $i < -N$ for some $N$. Dualizing, we obtain a complex $K^\bullet := \calHom_{\calO_{\PP^n}}(G^\bullet, \calO_{\PP^n})$ with the following properties:
\begin{enumerate}[\rm(i)]
\item Each $K^i$ is a finite direct sum of copies of twists of $\calO_{\PP^n}$.
\item $K^i = 0$ for $i < 0$ or $i > N$.
\item The complex $K^\bullet$ is quasi-isomorphic to $R\calHom_{\calO_{\PP^n}}(\calF^\vee,\calO_{\PP^n})$.
\end{enumerate}

In particular, since $\calF$ and $\calF^\vee$ are vector bundles over $U$, we have
\[
\calF|_U \simeq R\calHom_{\calO_{\PP^n}}(\calF^\vee,\calO_{\PP^n}))|_U,
\]
and thus
\[
K^\bullet \otimes_{\calO_{\PP^n}} \calO_{X_t} \simeq \calF \otimes_{\calO_{\PP^n}} \calO_{X_t} =: \calF_t
\]
for all $t$. Now (i) above and the previous case imply that
\[
\ilim H^i(X_t, K^a \otimes_{\calO_{\PP^n}} \calO_{X_t})
\]
has finite rank for each $a$ and each $i < d$. By Lemma~\ref{lemma:proj:system:complex}, this implies that the complex
\[
R\ilim R\Gamma(X_t, K^a \otimes_{\calO_{\PP^n}} \calO_{X_t})
\]
has $H^i$ of finite rank for $i < d$. Using (ii) and a standard spectral sequence, this implies that
\[
R\ilim R\Gamma(X_t, K^\bullet \otimes_{\calO_{\PP^n}} \calO_{X_t})
\]
has a finite rank $H^i$ for $i < d$. Using (iii), we deduce that
\[
R\ilim R\Gamma(X_t, \calF_t)
\]
has a finite rank $H^i$ for $i < d$. By Lemma~\ref{lemma:proj:system:complex}, this immediately shows that
\[
\ilim H^i(X_t, \calF_t)
\]
has finite rank for $i < d$, completing the proof.
\end{proof}

We used the following standard homological fact above:

\begin{lemma}
\label{lemma:proj:system:complex}
Fix a projective system $\{L_t\}$ of complexes of vector spaces over a field, and a cohomological degree $i$. Then there exists a Milnor exact sequence
\[
0 \to R^1 \ilim H^{i-1}(L_t) \to H^i(R\ilim L_t) \to \ilim H^i(L_t) \to 0.
\]
Moreover, if each $H^{i-1}(L_t)$ is finite rank, then the projective system $\{H^{i-1}(L_t)\}$ satisfies the Mittag-Leffler condition, and thus has a vanishing $R^1 \ilim$. 
\end{lemma}

The following result may be well-known to experts; however, for lack of a reference, we include it here as a corollary of Theorem~\ref{theorem:formal:finite}:

\begin{corollary}
\label{corollary:finite:rank}
Let $X$ be a closed subscheme of $\PP^n$ over a field $\FF$, and let~$\calF$ be a coherent sheaf on $\PP^n$ that is flat along $X$. Set $\calF_t=\calO_{X_t}\otimes_{\calO_{\PP^n}}\calF$. Suppose that either:
\begin{enumerate}[\quad\rm(1)]
\item the characteristic of $\FF$ is zero, and $X$ is lci, or
\item the characteristic of $\FF$ is positive, and $X$ is Cohen-Macaulay.
\end{enumerate}
Then, for each integer $k$ with $k<\dim X$, the inverse limit
\[
\ilim_t H^k(X_t,\calF_t)
\]
is a finite rank $\FF$-vector space. If $\calF=\calO_{\PP^n}(j)$ for $j<0$, this inverse limit vanishes.
\end{corollary}

The finite rank assertion need not hold even for $X$ of characteristic zero with Gorenstein rational singularities; see Example~\ref{example:singular}.

\begin{proof}
Setting $d=\dim X$ in the preceding theorem, it is enough to show that the local cohomology sheaves $\calH^i_X(\PP^n, \calO_{\PP^n})$ vanish for $i>n-d$. In case (1), this holds since for each homogeneous prime $\frakp$ with $I\subseteq\frakp\subsetneq\frakm$, the ideal $I_\frakp$ of $R_\frakp$ is generated by a regular sequence; in case~(2), the ideal $I_\frakp$ is perfect, so the vanishing holds by \cite[Proposition~III.4.1]{PS}.
\end{proof}

\begin{remark}
In the case where $\FF$ is the field of complex numbers $\CC$, the integer $\mu$ in~\eqref{equation:bass} is the $\CC$-rank of the singular cohomology group
\[
\Hsing^{n+i}(\CC^{n+1}\setminus V(I)\,;\,\CC),
\]
where $V(I)$ is the affine cone over $X$, see \cite[Theorem~3.1]{LSW}. When $X$ is moreover smooth, the integer $\mu$ may also be interpreted in terms of the ranks of the algebraic de~Rham cohomology groups~$H^k_\dR(X)$, see~\cite[Theorem~1.2]{Switala}.
\end{remark}

We now work towards the proof of Theorem~\ref{theorem:constant}; we first treat the case of $X$ smooth, and then turn to the more general case:

\subsection*{Smooth projective varieties}

Let $X$ be a closed subscheme of $\PP^n$, defined by a sheaf of ideals $\calI$, and let $X_t \subset \PP^n$ be the thickening of $X$ defined by $\calI^t$. The \emph{normal sheaf} of $X$ is the sheaf $\calN=(\calI/\calI^2)^\vee$, where $(-)^\vee$ denotes~$\calHom_{\calO_X}(-,\calO_X)$; this is a vector bundle when the subscheme $X$ is lci.

Following \cite{Hartshorne:ampleVB}, a vector bundle $\calE$ on $X$ is \emph{ample} if $\calO_{\pi}(1)$ is ample on the projective space bundle $\PP(\calE)$, where
\[
\pi\colon\PP(\calE)\to X
\]
is the projection morphism, and $\PP(\calE)$ parametrizes invertible quotients of $\calE$, i.e., we follow Grothendieck's convention regarding projective bundles. By the functoriality of projective space bundles, the property of being ample for a vector bundle passes to quotient bundles, \cite[Proposition~III.1.7]{Hartshorne:ample}, as well as to restrictions of $\calE$ to closed subschemes of~$X$, \cite[Proposition~III.1.3]{Hartshorne:ample}. In the case where $\calE$ is a line bundle, the map~$\pi$ is an isomorphism with~$\calO_\pi(1)$ corresponding to $\calE$, so the definition is consistent with that of an ample line bundle. The following example records the main source of ample bundles in our context:

\begin{example}
Let $X$ be a smooth projective subvariety of $\PP^n$. Since the tangent bundle~$\calT_{\PP^n}$ of $\PP^n$ is ample, the exact sequence
\[
0 \to \calT_X \to \calT_{\PP^n}\otimes\calO_X \to \calN \to 0
\]
shows that the normal bundle $\calN$ is ample, \cite[Proposition~III.2.1]{Hartshorne:ample}.
\end{example}

Ample vector bundles satisfy the following analogue of the Serre vanishing theorem, as well as Serre's cohomological criterion for ampleness:

\begin{lemma}
\label{lemma:ample}
Let $X$ be a proper variety, and let $\calE$ be a vector bundle on $X$. Then $\calE$ is ample if and only if for each coherent sheaf $\calF$ on $X$, and each $i > 0$, one has
\[
H^i(X,\Sym^t(\calE)\otimes\calF)=0 \quad\text{for } t \gg 0.
\]
\end{lemma}

\begin{proof}
This is \cite[Proposition~3.3]{Hartshorne:ampleVB}.
\end{proof}

In the case that $X$ is smooth, Theorem~\ref{theorem:constant} says the following:

\begin{theorem}[Hartshorne]
\label{theorem:smooth:constant}
Let $X$ be a smooth closed subvariety of $\PP^n$ over a field of characteristic zero. Let $\calF$ be a coherent sheaf on $\PP^n$ that is flat along $X$. Then, for each integer $k$ with~$k<\dim X$, the natural map
\[
H^k(X_{t+1},\calO_{X_{t+1}}\otimes_{\calO_{\PP^n}}\calF) \to H^k(X_t,\calO_{X_t}\otimes_{\calO_{\PP^n}}\calF)
\]
is an isomorphism for all $t\gg0$.
\end{theorem}

This appears in Hartshorne's proof of \cite[Theorem~8.1~(b)]{Hartshorne:ampleVB}, though it is not explicit in the statement of that theorem. We provide an argument for the convenience of the reader and then, in the following subsection, work through the modifications necessary for the case of lci subvarieties.

\begin{proof}
The cohomology groups have finite rank, so it suffices to prove that the maps are injective for $t\gg0$. In view of the cohomology exact sequence induced by
\[
0 \to \Sym^t(\calI/\calI^2)\otimes\calF \to \calO_{X_{t+1}}\otimes\calF \to \calO_{X_t}\otimes\calF \to 0,
\]
it suffices to prove that for each $k<\dim X$, one has
\[
H^k(X,\Sym^t(\calI/\calI^2)\otimes\calF)=0\quad\text{for }t\gg0.
\]
By Serre duality, this is equivalent to the vanishing of
\[
H^i(X,(\Sym^t(\calI/\calI^2))^\vee\otimes\calF^\vee\otimes K_X)\quad\text{ for }i>0\text{ and }t\gg0.
\]
Using $\Gamma^t$ for $t$-th divided powers, one has
\begin{equation}
\label{equation:symmetric}
(\Sym^t(\calI/\calI^2))^\vee = \Gamma^t(\calN) = \Sym^t(\calN),
\end{equation}
where the second equality in the above display uses the characteristic zero assumption, see for example~\cite[Proposition~A.2.7]{Eisenbud}. Thus, it suffices to prove that
\[
H^i(X,\Sym^t(\calN)\otimes\calF^\vee\otimes K_X)=0\quad\text{for }i>0 \text{ and }t\gg0,
\]
but this is immediate from Lemma~\ref{lemma:ample}, as $\calN$ is ample.
\end{proof}

\subsection*{Partially ample vector bundles, and local complete intersection singularities}

We now extend the previous results on smooth varieties to lci varieties. To carry this through, we need a weakening of the notion of an ample vector bundle that provides vanishing beyond a fixed cohomological degree, instead of vanishing in all positive cohomological degrees as in the smooth case. For line bundles, such a notion was studied in depth recently in \cite{TotaroAmple}. In the general case, the following definition captures the features relevant to our application, and was first studied in \cite{Arapura:AJM}:

\begin{definition}
Let $X$ be a projective scheme, and let $\calE$ be a vector bundle on $X$. Fix an integer $s$. We say that \emph{$\calE$ is $s$-ample} if for each coherent sheaf $\calF$ on $X$, and each integer $i$ with $i>s$, the group $H^i(X, \Sym^t(\calE) \otimes \calF)$ vanishes for $t \gg 0$.
\end{definition}

Thus, every bundle is $\dim X$-ample; by Lemma~\ref{lemma:ample}, a bundle is $0$-ample if and only if it is ample. It is also immediate that the restriction of an $s$-ample vector bundle to a closed subvariety remains $s$-ample. In the interest of keeping the paper self-contained, we spell out a few more elementary properties, even though they may be extracted from \cite{Arapura:AJM}. First, we often use the following criterion to check $s$-ampleness:

\begin{lemma}
\label{lemma:ample:criterion}
Let $X$ be a projective scheme, and let $\calE$ be a vector bundle on $X$. Fix an integer $s$, an ample line bundle $\calO_X(1)$ on $X$, and a line bundle $\calL$. Then $\calE$ is $s$-ample if and only if for each integer $j$, and each integer $i$ with $i > s$, one has
\[
H^i(X, \Sym^t(\calE)(j) \otimes \calL) = 0\quad\text{for }t \gg 0.
\]
\end{lemma}

\begin{proof}
The ``only if'' direction is clear. For the converse, let $\calF$ be a coherent sheaf. We will show that $H^i(X, \Sym^t(\calE) \otimes \calF) = 0$ for $i> s$ and $t \gg 0$ by descending induction on~$i$. If~$i > \dim X$, then the claim is clearly true as $X$ has cohomological dimension $\dim X$. For smaller $i$, pick a surjection
\[
\bigoplus_r \calO_{X}(j_r) \otimes \calL \onto \calF
\]
and let $\calG$ be its kernel. Tensoring with $\Sym^t(\calE)$ and taking the associated long exact cohomology sequence yields an exact sequence
\[
H^i(X, \oplus \Sym^t(\calE)(j_r) \otimes \calL) \to H^i(X, \Sym^t(\calE)\otimes \calF) \to H^{i+1}(X, \Sym^t(\calE) \otimes \calG).
\]
By the (descending) induction hypothesis, the third term vanishes for $t \gg 0$. The assumption on $\calE$ shows that the first term vanishes for $t \gg 0$ as well; thus, the middle term vanishes for $t \gg 0$, as desired.
\end{proof}

Next, we relate $s$-ample vector bundles to the corresponding notion for line bundles from \cite{TotaroAmple}, justifying the name:

\begin{lemma}
Let $X$ be a projective scheme, $s$ a nonnegative integer, $\calE$ a vector bundle, and $\pi \colon \PP(\calE) \to X$ the projection map. Then $\calE$ is an $s$-ample vector bundle if and only if~$\calO_{\pi}(1)$ is an $s$-ample line bundle on $\PP(\calE)$.
\end{lemma}

\begin{proof}
Assume first that $\calO_{\pi}(1)$ is $s$-ample on $\PP(\calE)$. Let $\calF$ be a coherent sheaf on $X$. Then we must show that $H^i(X, \Sym^t(\calE) \otimes \calF) = 0$ for $i > s$ and $t \gg 0$. But, if $t > 0$, then~$\Sym^t(\calE) = R^0 \pi_* \calO_{\pi}(t)$ and $R^i \pi_* \calO_{\pi}(t) = 0$ for $i > 0$. Thus, by the projection formula, we are reduced to showing that $H^i(\PP(\calE), \pi^*\calF(t)) = 0$ for $i > s$ and $t \gg 0$; this follows from the~$s$-ampleness of $\calO_\pi(1)$.

Conversely, assume that $\calE$ is $s$-ample. Fix an ample line bundle $\calL$ on $X$ such that
\[
\calM := \pi^*\calL \otimes \calO_\pi(1)
\]
is ample on $\PP(\calE)$. To show that $\calO_\pi(1)$ is $s$-ample on $\PP(\calE)$, by Lemma~\ref{lemma:ample:criterion}, it suffices to show that for fixed $j$, we have
\[
H^i(\PP(\calE), \calM^{\otimes j} \otimes \calO_\pi(t)) = 0
\]
for $i > s$ and $t \gg 0$. Now observe that
\[
\calM^{\otimes j} \otimes \calO_\pi(t) = \pi^* \calL^{\otimes j} \otimes \calO_\pi(t-j).
\]
Thus, by the projection formula, for $t > j$, we can write
\[
H^i(\PP(\calE), \calM^{\otimes j} \otimes \calO_\pi(t)) = H^i(X, \calL^{\otimes j} \otimes \Sym^{t-j}(\calE))
\]
which vanishes for $i > s$ and $t \gg j$ by the $s$-ampleness of $\calE$. This completes the proof.
\end{proof}

Using this, we can show that the property of $s$-ampleness passes to quotients:

\begin{lemma}
\label{lemma:quotient:partially:ample}
Let $X$ be a projective scheme, $s\ge0$ an integer, and $\calE$ an $s$-ample bundle on~$X$. Fix a quotient bundle $\calE \onto \calG$. Then $\calG$ is $s$-ample.
\end{lemma}

\begin{proof}
Write $\pi_{\calE} \colon \PP(\calE) \to X$ and $\pi_{\calG} \colon \PP(\calG) \to X$ for the projection morphisms. The quotient map $\calE \onto \calG$ induces a map $i\colon\PP(\calG) \to \PP(\calE)$ that is a closed immersion satisfying~$i^* \calO_{\pi_{\calE}}(1) = \calO_{\pi_{\calG}}(1)$. Thus, the claim reduces to the trivial observation that the restriction of $s$-ample line bundles to closed subschemes remains $s$-ample.
\end{proof}

In the sequel, it is useful to use the language of derived categories and derived tensor products to streamline proofs. Thus, from now on, for a scheme $X$, write $D(X)$ for the quasi-coherent derived category of $X$. We first record the following derived categorical observation concerning $s$-ample bundles:

\begin{lemma}
Let $X$ be a projective scheme, $s\ge0$ an integer, and $\calE$ an $s$-ample bundle on~$X$. Fix $K \in D^{\le 0}(X)$. Then
\[
H^i(X, \Sym^t(\calE) \otimes K) = 0
\]
for $i > s$ and $t \gg 0$.
\end{lemma}

\begin{proof}
Consider the triangle
\[
\tau^{\le s - \dim X} K \to K \to \tau^{> s - \dim X}(K).
\]
Then $M := \tau^{\le s - \dim X} K \otimes \Sym^t(\calE) \in D^{\le s - \dim X}(X)$, so $R\Gamma(X, M) \in D^{\le s}(\mathrm{pt})$ as $X$ has cohomological dimension $\dim X$, and thus $H^i(X, M) = 0$ for $i > s$. Thus, we may assume that~$K \simeq \tau^{> s - \dim X} K$ is bounded, and in $D^{\le 0}(X)$. For such $K$, the claim follows by induction on the number of nonzero cohomology sheaves and the definition of $s$-ampleness.
\end{proof}

To complete the proof of Theorem~\ref{theorem:constant}, we need a criterion for an extension of complexes to satisfy the cohomological criterion underlying $s$-ampleness. For this, and the sequel, it is convenient to use the theory of symmetric powers (as well as exterior and divided powers) for objects in $D(X)$. Such a theory was developed by Dold and Puppe via simplicial resolutions and the Dold-Kan correspondence as explained in, say, \cite[Chapter~1,~\S 4]{IllusieCC1}; for our applications, which concern characteristic zero schemes, a more elementary construction, as outlined in the next paragraph, suffices.

\begin{remark}
For most of the results in this paper, we work in characteristic zero. In this case, symmetric powers can be defined directly. Recall that for any commutative ring~$R$, there is a standard $\otimes$-structure defined on the category of complexes of $R$-modules as follows. If $K$ and $L$ are complexes, then $(K\otimes_RL)$ is the complex with
\[
(K\otimes_RL)^n=\bigoplus_{i+j=n}(K^i\otimes_RL^j),
\]
and the formula $d^n(x\otimes y)=d^i_K(x)\otimes y+(-1)^ix\otimes d^j_L(y)$ for $x\in K^i$ and $y\in L^j$ defines the differential $d^n\colon(K\otimes_RL)^n \to (K\otimes_RL)^{n+1}$. This $\otimes$-structure is symmetric, i.e., there is a canonical isomorphism
\[
\sigma_{K,L}\colon K\otimes_R L\to L\otimes_R K
\]
given by $x \otimes y \mapsto (-1)^{i \cdot j} y \otimes x$ for $x \in K^i$ and $y \in L^j$, and this isomorphism satisfies the identities necessary to induce a symmetric monoidal structure. In particular, for a complex~$K$ and integer $n \ge 0$, the $n$-fold tensor product $K^{\otimes_R n}$ acquires an $S_n$-action. Now if $R$ is a ring containing the rationals, then the (non-additive) functor $K \mapsto (K^{\otimes_R n})/S_n$ on chain complexes of free $R$-modules preserves quasi-isomorphisms (because if $M$ is an~$R$-module with an $S_n$-action, then the natural map $M\to M/S_n$ admits a splitting, namely the image of the map $M\to M$ that sends $m\in M$ to
\[
\frac{1}{n!}\sum_{\tau\in S_n}\tau(m)\in M
\]
is mapped isomorphically to $M/S_n$ via the natural surjection $M\to M/S_n$), and thus passes to the derived category to induce the sought for functor $\Sym^n\colon D(R) \to D(R)$. Likewise, there is a notion of exterior powers $\wedge^n\colon D(R) \to D(R)$, and one can directly show that~$\Sym^n(K[1]) \simeq \wedge^n K[n]$ for $K \in D(R)$.
\end{remark}

Our promised criterion is the following:

\begin{lemma}
\label{lemma:ample:triangle}
Let $X$ be a proper scheme, and let $K_1 \to K_2 \to K_3$ be an exact triangle in~$D(X)$. Fix an integer $s \ge 0$. Assume the following:
\begin{enumerate}[\quad\rm(1)]
\item $K_1$ is an $s$-ample vector bundle, viewed as a chain complex placed in degree $0$.
\item $K_3 \in D^{\le 0}(X)$, and the support of $\calH^0(K_3)$ has dimension at most $s$.
\end{enumerate}
If $\calF$ is a vector bundle on $X$, then, for each $i>s$, one has
\[
H^i(X,\Sym^t(K_2)\otimes\calF)=0 \quad\text{for } t \gg 0.
\]
\end{lemma}

We employ the following convention (as well as some obvious variants): given a triangulated category $\calC$, an object $K \in \calC$, and a set $S$ of objects of $\calC$, we say that \emph{$K$ admits a finite filtration with graded pieces in $S$} if there exists a chain of maps
\[
0 \simeq K_0 \to K_1 \to K_2 \to \cdots \to K_n \simeq K
\]
in $\calC$ such that the cone of each $K_i \to K_{i+1}$ is in $S$. The main relevant property for us is the following: if $H^*$ is a cohomological functor on $\calC$, then to show that~$H^i(K) = 0$ for some fixed $i$, it suffices to check that $H^i(F) = 0$ for all $F \in S$.

\begin{proof}
We claim $\Sym^t(K_2)$ admits a finite filtration in $D(X)$ with graded pieces of the form
\[
\Sym^a(K_1) \otimes \Sym^b(\calH^0(K_3)) \otimes \Sym^c(\tau^{< 0} K_3),
\]
where $a,b,c$ are nonnegative integers with $a+b+c=t$. To see this, we view the triangles $K_1 \to K_2 \to K_3$ and $\tau^{< 0}(K_3) \to K_3 \to \calH^0(K_3)$ as providing a lift $\tilde{K_2}$ of~$K_2$ to the filtered derived category with $\mathrm{gr}^2(\tilde{K_2}) = K_1$, $\mathrm{gr}^1(\tilde{K_2}) = \tau^{< 0}(K_3)$, and $\mathrm{gr}^0(\tilde{K_2}) = \calH^0(K_3)$. Applying the symmetric power functor on the filtered derived category then gives the induced filtration on $\Sym^t(K_2)$ by \cite[Lemma~4.2.5]{IllusieCC1}. (Similar arguments will occur repeatedly in the sequel without further elaboration.)

Hence, it suffices to show that for any vector bundle $\calF$ on $X$, and $i > s$, one has
\[
H^i(X, \Sym^a(K_1) \otimes \Sym^b(\calH^0(K_3)) \otimes \Sym^c(\tau^{< 0} K_3) \otimes \calF) = 0
\]
for $a+b+c \gg 0$. Note first that all complexes in sight lie in $D^{\le 0}(X)$. Hence, if $b > 0$, then the complex
\[ \Sym^a(K_1) \otimes \Sym^b(\calH^0(K_3)) \otimes \Sym^c(\tau^{< 0} K_3) \otimes \calF\]
lies in $D^{\le 0}(X)$ and is supported on a closed subset of dimension $\le s$; thus, it has vanishing cohomology in degree $> s$. We are therefore reduced to showing that for any vector bundle~$\calF$ on $X$, and $i > s$, one has
\[
H^i(X, \Sym^a(K_1) \otimes \Sym^c(\tau^{< 0} K_3) \otimes \calF) = 0
\]
for $a + c \gg 0$. Since $\tau^{< 0 } K_3 \in D^{\le -1}(X)$, we have $\Sym^c(\tau^{< 0} K_3) \in D^{\le -c}(X)$. Since~$K_1$ and~$\calF$ are vector bundles, this shows that
\[
\Sym^a(K_1) \otimes \Sym^c(\tau^{< 0} K_3) \otimes \calF \in D^{\le -c}(X).
\]
As $R\Gamma(X,-)$ carries $D^{\le -c}(X)$ to $D^{\le \dim X - c}(\mathrm{pt})$, there is nothing to prove for $c \ge \dim X$. Thus, we are reduced to showing that for any vector bundle~$\calF$ on $X$, and $i > s$, one has
\[
H^i(X, \Sym^a(K_1) \otimes \Sym^c(\tau^{< 0} K_3) \otimes \calF) = 0
\]
for $a + c \gg 0$ and $0 \le c \le \dim X$. As $\Sym^c(\tau^{< 0} K_3) \otimes \calF$ runs through only finitely many values in $D^{\le 0}(X)$ for $0 \le c \le \dim X$, the claim follows from the $s$-ampleness in $K_1$.
\end{proof}

Using this criterion, we arrive at one of the main results of this section:

\begin{theorem}
\label{theorem:normal:lci}
Let $X$ be a closed lci subvariety of $\PP^n$. If the singular locus of $X$ has dimension at most $s$, then the normal bundle $\calN$ is $s$-ample.
\end{theorem}

\begin{proof}
If $X$ is non-reduced, there is nothing to prove as $s = \dim X$. Thus, we may assume that~$X$ is reduced. We then have an exact sequence
\[
0 \to \calI/\calI^2 \to \Omega^1_{\PP^n}|_X \to \Omega^1_X \to 0.
\]
It follows that $\Omega^1_X$ has projective dimension at most $1$ locally on $X$, so $\underline{\R\Hom}(\Omega_X^1,\calO_X)$ lives in $D^{\le 1}(X)$. Dualizing the previous sequence and shifting gives an exact triangle
\[
\calT_{\PP^n}|_X \to \calN \to \underline{\R\Hom}(\Omega^1_X,\calO_X)[1],
\]
where $K_1 := \calT_{\PP^n}|_X$ is an ample vector bundle, and $K_3 := \underline{\R\Hom}(\Omega^1_X,\calO_X)[1]$ lies in $D^{\le 0}(X)$ with $\calH^0(K_3)$ supported on the singular locus of $X$ (which has dimension $\le s$). Applying Lemma~\ref{lemma:ample:triangle} then completes the proof.
\end{proof}

\begin{proof}[Proof of Theorem~\ref{theorem:constant}]
Using that the normal bundle is $s$-ample, the argument is the same as in the proof of Theorem~\ref{theorem:smooth:constant}.
\end{proof}

\begin{remark}
Theorem~\ref{theorem:normal:lci} yields a construction of ample vector bundles: if $X \subset \PP^n$ has isolated lci singularities, then its normal bundle $\calN$ is ample. If the singularities of $X$ are not isolated, one cannot expect ampleness of the normal bundle as Example~\ref{example:normal:not:ample} shows.
\end{remark}

\begin{remark}
\label{remark:effective:constancy}
Fix an integer $j$, and a closed lci variety $X$ of dimension $d$ in $\PP^n$ over a field of characteristic zero. Using the results proven in \S \ref{section:kodaira}, one can find an explicit integer~$t_0 = t_0(j)$ such that for all $t \ge t_0$, the map 
\[
H^k(X_{t+1},\calO_{X_{t+1}}(j)) \to H^k(X_t,\calO_{X_t}(j))
\]
as in Theorem~\ref{theorem:constant} is an isomorphism (respectively, injective) if $k < \codim(\Sing X) - 1$ (respectively, if~$k = \codim(\Sing X) - 1$). For this, following the proof of Theorem~\ref{theorem:smooth:constant}, we must find an integer $t_0$ such that 
\[
H^k(X, \Sym^t(\calI/\calI^2)(j)) = 0
\]
for $k < \codim(\Sing X)$ and $t \ge t_0$. Following the proof of Theorem~\ref{theorem:kodaira:lci} below, we see that the group above is filtered by groups that are subquotients of 
\[
H^{k-a}(X, \wedge^a L_X(-b+j)) \quad \mathrm{and} \quad H^{k-a-1}(X, \wedge^a L_X(-b+1+j))
\]
for $a,b \ge 0$ and $a+b = t$. By Lemma~\ref{lemma:local:coh:cc}, these groups vanish for $a>d$: the corresponding groups on $U := X - \Sing X$ vanish since $\Omega^a_U = 0$ for $a > d$ by the smoothness of $U$. By Theorem~\ref{theorem:kan:lci}, these groups also vanish if $b > j+1$ and $k < \codim(\Sing X)$. Thus, if we set
\[
t_0 = d+j+2,
\]
then the pigeonhole principle shows that both the groups above vanish if $t \ge t_0$. 
\end{remark}

%%%%%%%%%%%%%%%%%%%%%%%%%%%%%%%%%%%%%%%%%%%%%%%%%%%
\section{The Kodaira vanishing theorem for thickenings}
\label{section:kodaira}
%%%%%%%%%%%%%%%%%%%%%%%%%%%%%%%%%%%%%%%%%%%%%%%%%%%

Our goal in this section is to prove the following version of Kodaira vanishing for thickenings of lci subvarieties of projective space:

\begin{theorem}
\label{theorem:kodaira:lci}
Let $X$ be a closed lci subvariety of $\PP^n$ over a field of characteristic zero. Then for each $k$ with $k < \codim(\Sing X)$, and $j>0$, we have
\[
H^k(X_t, \calO_{X_t}(-j)) = 0.
\]
\end{theorem}

For $t = 1$, this is exactly the Kodaira vanishing theorem; see \cite{Arapura:PAMS} for more on the singular case, including an example illustrating the sharpness of the bound on cohomological degrees. We prove the statement for larger $t$, when $X$ is smooth, using the Kodaira-Akizuki-Nakano (KAN) vanishing theorem, see \cite[\S 6.4]{EsnaultViehweg}; the relevance of the KAN vanishing theorem here is that the kernel $\calI^t/\calI^{t+1} \simeq \Sym^t(\calI/\calI^2)$ of the map $\calO_{X_{t+1}} \to \calO_{X_{t}}$ receives a natural and cohomologically interesting map from $\Omega^t_X[-t]$ in $D(X)$.

To carry out the preceding strategy for potentially singular $X$, we must therefore establish a suitable extension of the KAN vanishing theorem to the singular case, and this extension is one of the key results of this section. To formulate this, we need a homologically well-behaved version of the sheaf of $t$-forms $\Omega^t_X$ in the singular case. One such version involves the Deligne-Du-Bois complexes $\underline{\Omega}^t_X$ provided by Deligne's mixed Hodge theory; the corresponding version of KAN vanishing was established by Kov{\'a}cs in \cite{KovacsKAN}. However, this is not sufficient for our purposes as the relation between the complexes $\underline{\Omega}^t_X$ and the bundles $\Sym^t(\calI/\calI^2)$ seems rather weak. Instead, our strategy is to replace $\Omega^t_X$ with the \emph{derived} exterior power $\wedge^t L_X$ of the cotangent complex. If $X$ is reduced, then $L_X \simeq \Omega^1_X$. Nevertheless, we need not have $\wedge^p L_X \simeq \Omega^p_X$ (i.e., the left term may have cohomology in negative degrees), and thus we do not know if Theorem~\ref{theorem:kan:lci} is true with $\Omega^b_X$ in place of~$\wedge^b L_X$. Replacing $\Omega^t_X$ is quite natural, since $\calI/\calI^2$ is itself a shifted cotangent complex: $L_{X/\PP^n} \simeq \calI/\calI^2[1]$. With these changes, our version of KAN vanishing takes the following shape: 

\begin{theorem}
\label{theorem:kan:lci}
Let $X$ be a closed lci subvariety of $\PP^n$ over a field of characteristic zero.~Then
\[
H^a(X, \wedge^b L_X(-j)) = 0
\]
for $a+b < \codim(\Sing X)$ and $j>0$. 
\end{theorem}

If $X$ is smooth, this result coincides with the original KAN vanishing theorem. Our strategy for proving Theorem~\ref{theorem:kan:lci} is to reduce to the smooth case by passage to a resolution; more precisely, we invoke Steenbrink vanishing on the resolution. To see that the pullback to the resolution is injective in the relevant range, we will check that the relevant cohomology survives on the nonsingular locus of $X$, and thus \emph{a fortiori} on a suitable resolution; this maneuver is ultimately the source of the ``$\codim(\Sing X)$" term in Theorem~\ref{theorem:kan:lci}. 

\begin{remark}
Theorem~\ref{theorem:kan:lci} is valid even when $X$ is not reduced: in this case, the singular locus has codimension $0$, so the theorem says that $H^a(X, \wedge^b L_X(-j))=0$ for $a + b < 0$ and~$j > 0$. Unlike in the smooth case, this assertion has content: the complex $\wedge^b L_X$ may have cohomology sheaves in negative degree, and thus hypercohomology in negative degree as well. To prove this assertion, note that Lemma~\ref{lemma:tor:dim:cc:lci} ensures that $\wedge^b L_X(-j) \in D^{\ge -b}$. Thus, if $a + b < 0$, then $a < -b$, so~$H^a(X, \wedge^b L_X(-j)) = 0$. 
\end{remark}

\begin{remark}
When $X$ is smooth, using Serre duality, the KAN vanishing theorem can be formulated as follows: Set $d=\dim X$. Then $H^a(X, \wedge^b L_X(j)) = 0$ for $a + b > d$ and~$j > 0$; this reformulation uses the identification $\wedge^b L_X \simeq (\wedge^{d-b} L_X)^\vee \otimes K_X$, which is not available in the singular case (essentially because $K_X \simeq \det(L_X)$ is \emph{not} $\wedge^d L_X$). Due to this discrepancy, the preceding Serre dual formulation fails dramatically in the singular case, as the following example illustrates:

Let $X \subset \PP^n$ be a projective variety of dimension $d > 0$ with a single isolated lci singularity at $x$. Then $\wedge^{d+1} L_X$ is supported at $x$ (since $X - \{x\}$ is smooth of dimension $d$), and~$\calH^0(\wedge^{d+1} L_X) \simeq \Omega^{d+1}_X \neq 0$ (since $\Omega^1_X$ is generated by $\ge d+1$ elements near $x$). Thus, it follows that $H^0(X, \wedge^{d+1} L_X(j)) \neq 0$ for any $j$, which gives the desired example.
\end{remark}

\begin{remark}
The lci assumption in Theorem~\ref{theorem:kan:lci} is necessary. Indeed, if $X \subset \PP^n$ has an isolated singularity at $x \in X$, and if $X$ is not lci at $x$, then $\calH^i(L_X) \neq 0$ for arbitrarily negative integers~$i$, see \cite{Avramov:Quillen}. Since these sheaves are supported at $x$, we immediately deduce that for any $j$, we have $H^k(X, L_{X}(-j)) \neq 0$ for arbitrarily negative $k$. 
\end{remark}

For the rest of this section, we fix a closed lci subvariety $X \subset \PP^n$ of dimension $d > 0$ over a field of characteristic $0$. Let $Z$ be the singular locus of $X$, with complement $U := X - Z$, and set $s = \dim Z$. Write $j\colon U \to X$ for the resulting open immersion. As usual, we write~$\calI/\calI^2$ for the conormal bundle of $X$, and $\calN = (\calI/\calI^2)^\vee$ for the normal bundle. 

Since $U$ is smooth, we have ${L_X}|_U \simeq L_U \simeq \Omega^1_U$. Applying $- \otimes \wedge^b L_X$ to $\calO_X \to Rj_* \calO_U$ induces a canonical map $\wedge^b L_X \to R j_* \Omega^b_U$ and consequently a canonical map
\[
H^a(X, \wedge^b L_X(j)) \to H^a(U, \Omega^b_U(j)).
\]
One of the key ingredients of the proof of Theorem~\ref{theorem:kan:lci} is the injectivity of
\[
H^a(X, \wedge^b L_X(j)) \to H^a(U, \Omega^b_U(j))
\]
which is proved in Lemma~\ref{lemma:local:coh:cc}. Before proving Lemma~\ref{lemma:local:coh:cc}, we record two lemmas.

\begin{lemma}
\label{lemma:tor:dim:cc:lci}
The complex $\wedge^b L_X$ has $\Tor$-dimension $\le b$ locally on $X$.
\end{lemma}

\begin{proof}
There is nothing to show for $b = 0$. The exact triangle
\[
\Omega^1_{\PP^n}|_X \to L_X \to \calI/\calI^2[1]
\]
verifies the claim for $b = 1$. In general, taking wedge powers of the previous triangle, shows that $\wedge^b L_X$ admits a finite filtration in $D(X)$ with graded pieces of the form 
\[
\Omega^i_{\PP^n}|_X \otimes \wedge^j(\calI/\calI^2[1]) \simeq \Omega^i_{\PP^n|_X} \otimes \Gamma^j(\calI/\calI^2)[j]
\]
for $i + j = b$ and $i,j \ge 0$. This immediately gives the claim. 
\end{proof}

\begin{lemma}
\label{lemma:local:coh:cm}
We have $\underline{R\Gamma}_Z(\calO_X) \in D^{\ge d-s}(X)$. 
\end{lemma}

\begin{proof}
Translating to local terms, and using that $X$ is Cohen-Macaulay, one needs to verify the following: if $R$ is a Cohen-Macaulay ring, and $I$ an ideal, then $H^k_I(R) = 0$ for integers~$k$ with $k < \height I$. But this is immediate from the Cohen-Macaulay property.
\end{proof}

Combining the previous lemmas, we obtain the promised injectivity:

\begin{lemma}
\label{lemma:local:coh:cc}
The homotopy-kernel $K$ of the canonical map 
\[
\wedge^b L_X \to R j_* \Omega^b_U
\]
lies in $D^{\ge d-s-b}(X)$. Consequently, the canonical map 
\[
H^a(X, \wedge^b L_X(j)) \to H^a(U, \Omega^b_U(j))
\]
is injective for $a < d-s-b$ and any $j$.
\end{lemma}

\begin{proof}
The homotopy-kernel of $\calO_X \to Rj_* \calO_U$ is, by definition, the local cohomology complex $\underline{R\Gamma}_Z(\calO_X)$, which lies in $D^{\ge d-s}(X)$ by Lemma~\ref{lemma:local:coh:cm}. Hence, after tensoring with $\wedge^b L_X$, we see that the homotopy-kernel $K$ in question is identified with $\wedge^b L_X \otimes \underline{R\Gamma}_Z(\calO_X)$; this lies in~$D^{\ge d-s-b}(X)$ as $\wedge^b L_X$ has $\Tor$-dimension $\le b$ by Lemma~\ref{lemma:tor:dim:cc:lci}, proving the first part. The second part is immediate since the functor $K \mapsto R\Gamma(X,K(j))$ carries $D^{\ge i}$ to $D^{\ge i}$ for each $i$ and any $j$.
\end{proof}

To proceed further along the lines sketched above, choose a resolution $\pi\colon Y \to X$, i.e.,~$\pi$ is a proper birational map with $Y$ smooth, $\pi^{-1}(U) \simeq U$, and $\pi^{-1}(Z) =: E$ is a simple normal crossings divisor. We need the following form of Steenbrink vanishing:

\begin{theorem}[Steenbrink]
\label{theorem:steenbrink}
We have 
\[
H^a(Y, \Omega^b_Y(\log E) \otimes \pi^* \calO_X(-j)) = 0
\]
for $a+b < d$ and $j > 0$. 
\end{theorem}

This result is remarkable for the following reason: it is known that, unlike the Kodaira vanishing theorem, the KAN vanishing theorem fails for line bundles that are merely semiample and big (like $\pi^* \calO_X(1)$). The previous theorem thus says that the introduction of log poles gets around this problem. This is proven in \cite[Theorem~2~(a$'$)]{Steenbrink}. Alternatively, it may be deduced from Esnault-Viehweg's logarithmic KAN vanishing theorem as follows:

\begin{proof}
We apply \cite[Theorem~6.7]{EsnaultViehweg} with the following notational changes: the roles of $X$ and $Y$ are reversed, and $\calA = \calO_X(j)$. To get our desired conclusion, we need the quantity $r(\tau|_{Y-E})$ in \emph{loc.~cit.} to be $0$. However, this is immediate from the definition of $r$ (see \cite[\S 4.10]{EsnaultViehweg}) and the fact that the induced map $Y-E \to X-E$ is an an isomorphism. 
\end{proof}

We can now complete the proof of Theorem~\ref{theorem:kan:lci} as follows:

\begin{proof}[Proof of Theorem~\ref{theorem:kan:lci}]
The open immersion $j\colon U \to X$ factors as $U \stackrel{h}{\to} Y \stackrel{\pi}{\to} X$ by construction. As the divisor $E$ on $Y$ restricts to the trivial divisor on $U$, since $E \cap U = \emptyset$, we obtain induced maps
\[
\wedge^b L_X \to R \pi_* (\Omega^b_Y(\log E)) \to R j_* \Omega^b_U,
\]
and hence also maps
\[
\wedge^b L_X(-j) \to R \pi_* (\Omega^b_Y(\log E))(-j) \to R j_* \Omega^b_U(-j)
\]
by twisting. Now, if $j > 0$, the composite map is injective on $H^a(X,-)$ for $a + b < d-s$ by Lemma~\ref{lemma:local:coh:cc}, and hence the same is true for the first map. On the other hand, 
\[
H^a(X, R \pi_* \Omega^b_Y(\log E)(-j)) \simeq H^a(Y, \Omega^b_Y(\log E) \otimes \pi^* \calO_X(-j)),
\]
by the projection formula. The latter vanishes for $j > 0$ and $a + b < d$ by Theorem~\ref{theorem:steenbrink}. Combining these facts, we obtain the conclusion.
\end{proof}

To move from Theorem~\ref{theorem:kan:lci} to Theorem~\ref{theorem:kodaira:lci}, it is convenient to use the following:

\begin{lemma}
\label{lemma:omega:Pn}
For each $t \ge 1$, there exists an exact triangle
\[
\calO_{\PP^n}(-t+1)^{\oplus r}[-1] \to \Sym^t(\Omega^1_{\PP^n}) \to \calO_{\PP^n}(-t)^{\oplus s}
\]
for suitable integers $r,s \ge 1$.
\end{lemma}

\begin{proof}
The shifted Euler sequence on $\PP^n$ is given by
\[
\calO_{\PP^n}[-1] \to \Omega^1_{\PP^n} \to \calO_{\PP^n}(-1)^{\oplus (n+1)}.
\]
Taking symmetric powers, and observing that $\Sym^t(\calO_{\PP^n}[-1]) \simeq \wedge^t \calO_{\PP^n}[-t]$ vanishes for integers~$t \ge 2$, we obtain an exact triangle
\[
\Sym^{t-1}(\calO_{\PP^n}(-1)^{\oplus (n+1)})[-1] \to \Sym^t(\Omega^1_{\PP^n}) \to \Sym^t(\calO_{\PP^n}(-1)^{\oplus (n+1)}).
\]
The lemma now follows immediately from the behavior of $\Sym$ under direct sums.
\end{proof}

\begin{proof}[Proof of Theorem~\ref{theorem:kodaira:lci}]
The kernel of the surjection $\calO_{X_{t+1}} \to \calO_{X_t}$ is given by $\Sym^t(\calI/\calI^2)$, so it suffices to show that 
\[
H^k(X, \Sym^t(\calI/\calI^2)(-j)) = 0
\]
for $k < d -s$ and $j, t > 0$. The shifted transitivity triangle for $X \into \PP^n$ takes the form
\[
L_X[-1] \to \calI/\calI^2 \to \Omega^1_{\PP^n}|_X.
\]
Thus, $\Sym^t(\calI/\calI^2)$ admits a finite filtration in $D(X)$ with graded pieces given by
\[
M_{a,b} := \Sym^a(\Omega^1_{\PP^n}|_X) \otimes \Sym^b(L_X[-1]) \simeq \Sym^a(\Omega^1_{\PP^n}|_X) \otimes \wedge^b L_X[-b]
\]
for $a,b \ge 0$ and $a + b = t > 0$. It suffices to show that for each such pair $(a,b)$, we have
\[
H^k(X, M_{a,b}(-j)) = 0
\]
for $k < d-s$ and $j > 0$. If $a = 0$, then $b > 0$, and then the claim follows from Theorem~\ref{theorem:kan:lci}. For $a > 0$, by Lemma~\ref{lemma:omega:Pn}, we obtain an exact triangle
\[
\calO_X(-a+1)^{\oplus r}[-1] \otimes \wedge^b L_X[-b](-j) \to M_{a,b}(-j) \to \calO_X(-a)^{\oplus s} \otimes \wedge^b L_X[-b](-j)
\]
for suitable integers $r,s > 0$. This simplifies to
\[
\big(\wedge^b L_X(-j-a+1)\big)^{\oplus r}[-b-1] \to M_{a,b}(-j) \to \big(\wedge^b L_X(-j-a)\big)^{\oplus s}[-b].
\]
The outer terms have no $k$-th cohomology for $k < d-s$ by Theorem~\ref{theorem:kan:lci}, so neither does the middle term.
\end{proof}

\begin{remark}
It is not reasonable to expect Kodaira vanishing for arbitrary thickenings, even in the smooth case; see Examples~\ref{example:2x3:thickening} and~\ref{example:abelian:thickening}.
\end{remark}

%%%%%%%%%%%%%%%%%%%%%%%%%%%%%%%%%%%%%%%%%%%%%%%%%%%
\section{Injectivity of maps on $\Ext$ modules}
\label{section:EMS}
%%%%%%%%%%%%%%%%%%%%%%%%%%%%%%%%%%%%%%%%%%%%%%%%%%%

Let $I$ be an ideal of a commutative noetherian ring $R$. Since the local cohomology module $H^k_I(R)$ can be expressed as the direct limit
\[
\dlim_t\Ext^k_R(R/I^t,R),
\]
it follows immediately that $H^k_I(R)$ equals $\dlim_t\Ext^k_R(R/I_t,R)$, where $\{I_t\}$ is any decreasing chain of ideals that is cofinal with the chain $\{I^t\}$. In this context, Eisenbud, Musta{\c{t}}{\u{a}}, and Stillman raised the following:

\begin{question}\cite[Question~6.1]{EMS}\label{Q1}
Let $R$ be a polynomial ring. For which ideals $I$ of $R$ does there exist a chain $\{I_t\}$ as above, such that for each $k,t$, the natural map
\[
\Ext^k_R(R/I_t,R)\to H^k_I(R)
\]
is injective?
\end{question}

\begin{question}\cite[Question~6.2]{EMS}\label{Q2}
Let $R$ be a polynomial ring. For which ideals $I$ of $R$ is the natural map $\Ext^k_R(R/I,R)\to H^k_I(R)$ an inclusion?
\end{question}

Question~\ref{Q1} is motivated by the fact that $H^k_I(R)$ is typically not finitely generated as an~$R$-module, whereas the modules $\Ext^k_R(R/I_t,R)$ are finitely generated; consequently, an affirmative answer to Question~\ref{Q1} yields a filtration of $H^k_I(R)$ by the readily computable modules~$\Ext^k_R(R/I_t,R)$. We record some partial results:

\begin{enumerate}[\rm(i)]
\item\label{remark:depth}
If $k$ equals the length of a maximal $R$-regular sequence in $I$, then the natural map
\[
\Ext^k_R(R/I,R)\to H^k_I(R)
\]
is injective. This is readily proved using a minimal injective resolution of $R$, see for example \cite[Remark~1.4]{SW}.

\item Suppose $I$ is a set-theoretic complete intersection, i.e., $f_1,\dots,f_n$ is a regular sequence generating $I$ up to radical, then the ideals $I_t=(f_1^t,\dots,f_n^t)$ form a descending chain, cofinal with $\{I^t\}$, such that the natural maps
\[
\Ext^k_R(R/I_t,R)\to H^k_I(R)
\]
are injective for each $k\ge0$ and $t\ge1$: when $k=n$, this follows from~(\ref{remark:depth}), whereas if~$k\neq n$, then $\Ext^k_R(R/I_t,R)=0=H^k_I(R)$.

\item\label{remark:monomial}
Let $R$ be a polynomial ring over a field. Suppose an ideal $I$ is generated by square-free monomials $m_1,\dots,m_r$, set $I^{[t]}=(m_1^t,\dots,m_r^t)$ for $t\ge1$. Then, by \cite[Theorem~1~(i)]{Lyubeznik-monomial}, the natural maps
\[
\Ext^k_R(R/I^{[t]},R) \to H^k_I(R)
\]
are injective for each $k\ge0$ and $t\ge1$; see also \cite[Theorem~1.1]{Mustata}.

\item\label{remark:fpure}
Suppose $R$ is a regular ring containing a field of characteristic $p>0$, and $I$ is an ideal such that $R/I$ is $F$-pure. Set $I^{[p^t]}=(a^{p^t}\mid a\in I)$. Then the natural maps
\[
\Ext^k_R(R/I^{[p^t]},R)\to H^k_I(R)
\]
are injective for each $k\ge0$ and $t\ge0$; this is \cite[Theorem~1.3]{SW}.

\item Let $R$ be a regular ring and $I$ an ideal of $R$. Suppose $R$ has a flat endomorphism $\phi$ such that $\{\phi^t(I)R\}_{t\ge0}$ is a decreasing chain of ideals cofinal with $\{I^t\}_{t\ge0}$, and the induced endomorphism $\bar{\phi}\colon R/I\to R/I$ is pure. Then, for all $k\ge 0$ and $t\ge 0$, the natural map
\[
\Ext^k_R(R/\phi^t(I)R,R)\to H^k_I(R)
\]
is injective, \cite[Theorem~2.8]{SW}. This recovers (\ref{remark:fpure}) by taking $\phi$ to be the Frobenius endomorphism of $R$, and (\ref{remark:monomial}) by taking the endomorphism of the polynomial ring that raises each variable to a power.

\item\label{remark:RWW}
Let $X=(x_{ij})$ be an $m\times n$ matrix of indeterminates with $m\le n$, and $R$ be the ring of polynomials in $x_{ij}$ with coefficients in a field of characteristic zero. Let $I$ be the ideal generated by the size $m$ minors of $X$, i.e., the \emph{maximal minors}. Then the natural map
\[
\Ext^k_R(R/I^t,R) \to H^k_I(R)
\]
is injective for each $k\ge0$ and $t\ge1$ by~\cite[\S 4]{RaicuWeymanWitt}, see also~\cite[Theorem~4.2]{RaicuWeyman}. The injectivity may fail for non-maximal minors, Example~\ref{example:nonmax}, and in the case of positive characteristic, Example~\ref{example:charp}. The recent paper \cite{Raicu} contains further injectivity results as well as a Kodaira vanishing theorem for thickenings of determinantal ideals.

\item Corollary~\ref{corollary:ext} of the present paper provides a partial answer to Question~\ref{Q2} after restricting to an arbitrary but fixed graded component.
\end{enumerate}

%%%%%%%%%%%%%%%%%%%%%%%%%%%%%%%%%%%%%%%%%%%%%%%%%%%
\section{Examples}
\label{section:examples}
%%%%%%%%%%%%%%%%%%%%%%%%%%%%%%%%%%%%%%%%%%%%%%%%%%%

In this section, we record examples illustrating our results, and primarily their sharpness. We first begin with an example illustrating Theorem~\ref{theorem:constant}; the main purpose of this example is to record some calculations that will be useful later.

\begin{example}
\label{example:2x3}
Let $R$ be the polynomial ring in a $2\times 3$ matrix of indeterminates over a field~$\FF$ of characteristic zero, and consider the ideal $I$ generated by the size two minors of the matrix. Note that $X=\Proj R/I$ equals $\PP^1\times\PP^2$ under the Segre embedding in $\PP^5$.

The local cohomology module $H^3_I(R)$ is isomorphic, as a graded module, to~$H^6_\frakm(R)$, see for example~\cite[Theorem~1.2]{LSW}. It follows that it has Hilbert series
\[
\sum_j \rank\, {H^3_I(R)}_j\,z^j = \frac{z^{-6}}{(1-z^{-1})^6} = z^{-6}+6z^{-7}+21z^{-8}+56z^{-9}+126z^{-10}+\cdots.
\]
The vector space ranks
\[
\rank{\Ext^3_R(R/I^t,R)}_j = \rank H^2(X_t,\calO_{X_t}(-6-j)),
\]
as computed by \emph{Macaulay2}, \cite{GS}, are recorded in the following table; entries that are $0$ are omitted. The last row records the stable value attained along each column, namely the rank of $H^3_I(R)_j$, as computed from the Hilbert series displayed above.

\begin{table}[H]
\begin{center}
\begin{tabular}{|c||*{11}{c|}}
\hline
\backslashbox{$t$}{$j$} & $-6$ & $-7$ & $-8$ & $-9$ & $-10$ & $-11$ & $-12$ & $-13$ & $-14$ & $-15$ & $-16$\\
\hline
$1$ & \makebox[1.5em] & \makebox[1.5em] & \makebox[1.5em] & \makebox[1.5em] & \makebox[1.5em] & \makebox[1.5em] & \makebox[1.5em] & \makebox[1.5em] & \makebox[1.5em] & \makebox[1.5em] & \\
\hline
$2$ & $1$ & & & & & & & & & & \\
\hline
$3$ & $1$ & $6$ & $3$ & & & & & & & & \\
\hline
$4$ & $1$ & $6$ & $21$ & $16$ & $6$ & & & & & & \\
\hline
$5$ & $1$ & $6$ & $21$ & $56$ & $51$ & $30$ & $10$ & & & & \\
\hline
$6$ & $1$ & $6$ & $21$ & $56$ & $126$ & $126$ & $91$ & $48$ & $15$ & & \\
\hline
$7$ & $1$ & $6$ & $21$ & $56$ & $126$ & $252$ & $266$ & $216$ & $141$ & $70$ & $21$ \\
\hline
\hline
$\dlim_t$ & $1$ & $6$ & $21$ & $56$ & $126$ & $252$ & $462$ & $792$ & $1287$ & $2002$ & $3003$ \\[4pt]
\hline
\end{tabular}
\end{center}
\end{table}

Since $I$ is the ideal of maximal sized minors, \cite[\S 4]{RaicuWeymanWitt} implies that the map
\[
\Ext^3_R(R/I^t,R) \to \Ext^3_R(R/I^{t+1},R)
\]
is injective for each $t\ge 1$. Thus, for a fixed integer $j$, if the rank of ${\Ext^3_R(R/I^{t_0},R)}_j$ equals that of ${H^3_I(R)}_j$ for some $t_0$, then the displayed map is an isomorphism for each $t\ge t_0$.

The injectivity may fail for non-maximal minors, see Example~\ref{example:nonmax}.
\end{example}

Next, we show that the Kodaira vanishing theorem fails for arbitrary infinitesimal thickenings of smooth varieties. First, we give such an example that is not lci; in this case, the vanishing fails in all negative degrees.

\begin{example}
\label{example:2x3:thickening}
Let $R$ and $I$ as in Example~\ref{example:2x3}, e.g., set $R=\FF[u,v,w,x,y,z]$ where $\FF$ has characteristic zero, and $I=(\Delta_1, \Delta_2, \Delta_3)$ where
\[
\Delta_1=vz-wy,\quad \Delta_2=wx-uz,\quad \Delta_3=uy-vx.
\]
The ideal $J=(\Delta_1^2, \Delta_2, \Delta_3)$ has the same radical as $I$, i.e., $X=\Proj R/J$ has the same reduced structure as $\Proj R/I=\PP^1\times\PP^2$. We claim that
\[
H^2(X,\calO_X(j))\neq 0\quad\text{ for each }j\le 0.
\]
The projective resolution of $R/J$ is readily computed to be
\[
\CD
0@<<< R
@<{\begin{pmatrix}\Delta_1^2 & \Delta_2 & \Delta_3\end{pmatrix}}<<R^3
@<{\begin{pmatrix}0& u& x\\ -\Delta_3 & v\Delta_1 & y\Delta_1\\ \Delta_2& w\Delta_1& z\Delta_1\end{pmatrix}}<<R^3
@<{\begin{pmatrix}\Delta_1\\ -x\\ u\end{pmatrix}}<<R^1@<<<0,
\endCD
\]
from which it follows that
\[
\Ext^3_R(R/J,R) \cong \big(R/(\Delta_1,u,x)R\big)(6).
\]
Thus, $\Ext^3_R(R/J,R)$ has Hilbert series
\[
\sum_j \rank {\Ext^3_R(R/J,R)}_j\,z^j = \frac{z^{-6}+z^{-5}}{(1-z)^3},
\]
specifically, ${\Ext^3_R(R/J,R)}_j\neq0$ for each $j\ge-6$. The claim follows using Serre duality.
\end{example}

Next, we give an example of an lci thickening of an abelian variety where Kodaira vanishing also fails:

\begin{example}
\label{example:abelian:thickening}
Let $A$ be an abelian variety, and let $\calL \in \Pic(A)$ be an ample line bundle. Then $\calO_A \oplus \calL$ is an augmented $\calO_A$-algebra with $\calL$ given the structure of an ideal whose square is zero. Let $X := \underline{\Spec} A$. The $\calO_A$-algebra maps
\[
\calO_A \to \calO_A \oplus \calL \quad \mathrm{and} \quad \calO_A \oplus \calL \to \calO_A
\]
give a thickening $i\colon A \into X$ that admits a (finite) retraction $r\colon X \to A$ with~$r_* \calO_X = \calO_A \oplus \calL$. It is immediate that $X$ has lci singularities. Moreover, $\calM = r^* \calL$ is an ample line bundle on $X$, and one checks that $H^k(X,\calM^{-1}) \neq 0$ for $0 \le k \le \dim A$.
\end{example}

Recall that one of our motivations was Question~\ref{Q1} raised in \cite{EMS}; in the context covered by Theorem~\ref{theorem:constant}, we obtained a partial positive answer to this question: for suitable~$k$, the map ${\Ext^k_R(R/I_t,R)}_j\to {H^k_I(R)}_j$ is injective for $t \gg 0$, once we fix a degree $j$. The next example shows that one cannot expect to do better, i.e., one cannot expect uniformity in~$j$. Thus, our answer to Question~\ref{Q1}, in the context of standard graded rings covered by Theorem~\ref{theorem:constant}, is optimal.

\begin{example}
\label{example:nonmax}
Let $R$ be the ring of polynomials in a $3\times 3$ matrix of indeterminates over a field~$\FF$ of characteristic zero; set $I$ to be the ideal generated by the size two minors of the matrix. Then $X=\Proj R/I$ equals $\PP^2\times\PP^2$ under the Segre embedding in $\PP^8$. The ideal~$I$ has cohomological dimension $6$; in particular, $H^9_I(R)=0$.

We claim that $\Ext^9_R(R/I^t,R)$ is nonzero for each $t\ge 2$. The primary decomposition of powers of determinantal ideals is given by \cite[Corollary~7.3]{DEP}; in our case, for each~$t\ge2$, this provides the irredundant primary decomposition of $I^t$ as
\[
I^t = I^{(t)} \cap \frakm^{2t},
\]
where $\frakm$ is the homogeneous maximal ideal of $R$. It follows that for each $t\ge 2$, the maximal ideal is an associated prime of $R/I^t$, and hence that $\pd_R R/I^t=9$ by the Auslander-Buchsbaum formula, which proves the claim. The ranks of the vector spaces
\[
{\Ext^9_R(R/I^t,R)}_j
\]
are recorded in the following table:
\begin{table}[H]
\begin{center}
\begin{tabular}{|c||*{11}{c|}}
\hline
\backslashbox{$t$}{$j$} & $-12$ & $-13$ & $-14$ & $-15$ & $-16$ & $-17$ & $-18$ & $-19$ & $-20$ & $-21$ & $-22$ \\
\hline
$1$ & \makebox[1.5em] & \makebox[1.5em] & \makebox[1.5em] & \makebox[1.5em] & \makebox[1.5em] & \makebox[1.5em] & \makebox[1.5em] & \makebox[1.5em] & \makebox[1.5em] & \makebox[1.5em] & \\
\hline
$2$ & $1$ & & & & & & & & & & \\
\hline
$3$ & & & $9$ & & & & & & & & \\
\hline
$4$ & & & & $1$ & $45$ & & & & & & \\
\hline
$5$ & & & & & & $9$ & $165$ & & & & \\
\hline
$6$ & & & & & & & $1$ & $45$ & $495$ & & \\
\hline
$7$ & & & & & & & & & $9$ & $165$ & $1287$ \\
\hline
\end{tabular}
\end{center}
\end{table}

The table bears out the two facts that we have proved:
\begin{enumerate}[\quad\rm(i)]
\item each row other than the first is nonzero, since $\Ext^9_R(R/I^t,R)\neq0$ for $t\ge 2$, and

\item each column is eventually zero since the stable value, i.e., ${H^9_I(R)}_j$, is zero, and each column is eventually constant by Theorem~\ref{theorem:constant}.
\end{enumerate}
\end{example}

The next example illustrates why Theorem~\ref{theorem:constant} is limited to characteristic zero:

\begin{example}
\label{example:charp}
Consider the analog of Example~\ref{example:2x3} in characteristic $p>0$, i.e., $X$ is the Segre embedding of $\PP^1\times\PP^2$ in $\PP^5$, working over the field $\FF_p$. We claim that $H^2(X_t,\calO_{X_t})$ is nonzero for each $t\ge 2$, though the map
\begin{equation}
\label{eq:t:tprime}
H^2(X_{t'},\calO_{X_{t'}}) \to H^2(X_t,\calO_{X_t})
\end{equation}
is zero for all $t'\gg t\ge 1$.

Let $R$ be the ring of polynomials over $\ZZ$ in a $2\times 3$ matrix of indeterminates; set $I$ to be the ideal generated by the size two minors of the matrix, and $\frakm$ to be the ideal generated by the indeterminates. The $\ZZ$-module ${H^3_\frakm(R/I^t)}_0$ is finitely generated, see, for example, \cite[Theorem~III.5.2]{Hartshorne:ag}. Moreover, for $t\ge 2$, it is a rank one $\ZZ$-module, since
\[
{H^3_\frakm(R/I^t)}_0\otimes_\ZZ\QQ = \QQ
\]
by Example~\ref{example:2x3}. It follows that multiplication by $p$ is not surjective on ${H^3_\frakm(R/I^t)}_0$. The exact sequence
\[
\CD
{H^3_\frakm(R/I^t)}_0 @>p>> {H^3_\frakm(R/I^t)}_0 @>>> {H^3_\frakm(R/(pR+I^t))}_0
\endCD
\]
then shows that
\[
{H^3_\frakm(R/(pR+I^t))}_0 = H^2(X_t,\calO_{X_t})
\]
is nonzero, as claimed.

The determinantal ring $R/(pR+I)$ is Cohen-Macaulay, and of dimension $4$. It follows by the flatness of the Frobenius endomorphism of $R/pR$ that $R/(pR+I^{[p^e]})$ is Cohen-Macaulay for each $e\ge 1$, and so
\[
H^3_\frakm(R/(pR+I^{[p^e]})) = 0.
\]
For $t'\gg t\ge 1$, the map~\eqref{eq:t:tprime} factors through the degree zero component of this module, and hence must be zero.
\end{example}

Likewise, Theorem~\ref{theorem:constant} may fail for singular projective varieties:

\begin{example}
\label{example:singular}
Take $Y\subset\PP^6$ to be the projective cone over $\PP^1\times\PP^2\subset\PP^5$, working over a field of characteristic zero, i.e., the projective cone over $X$ as in Example~\ref{example:2x3}. Then $Y_t$ is the projective cone over $X_t$, and it follows that
\[
\ilim H^3(Y_t, \calO_{Y_t}) = \sum_{j\ge1} \ilim H^2(X_t, \calO_{X_t}(j)).
\]
The Serre dual of this is
\[
\dlim_t{\Ext^3_R(R/I^t,R)}_{\le-7} = {H^3_I(R)}_{\le-7} \cong {H^6_\frakm(R)}_{\le-7},
\]
which has infinite rank.

The variety $Y$ has an isolated singular point that is Cohen-Macaulay but not Gorenstein; for a similar example with an isolated rational singularity that is Gorenstein, take the projective cone over $\PP^2\times\PP^2\subset\PP^8$.
\end{example}

Theorem~\ref{theorem:normal:lci} shows that if $X \subset \PP^n$ is lci with a singular locus of dimension $s$, then the normal bundle is $s$-ample. We now give examples illustrating that one cannot do better. We first give an example in the case $s = \dim X$:

\begin{example}
\label{example:normal:not:ample}
Let $\PP^s \subset \PP^N$ be a linear subspace of codimension $c = N - s$ defined by an ideal $\calI \subset \calO_{\PP^N}$, so $\calI/\calI^2 \simeq \calO_{\PP^s}(-1)^c$. Choose a rank $c-1$ subbundle $\calE \subset \calI/\calI^2$ so that the quotient line bundle $\calQ=(\calI/\calI^2)/\calE$ has positive degree; this can be done if $c \gg 0$ by repeated use of the Euler sequence. Then the inverse image of $\calE$ in $\calI$ defines an ideal $\calJ \subset \calO_{\PP^N}$ with $\calI^2 \subset \calJ \subset \calI$, and $\calJ/\calI^2 = \calE$, and $\calI/\calJ = \calQ$. Set $X$ to be the closed subscheme of $\PP^N$ defined by $\calJ$. Thus, $X$ is an infinitesimal thickening of $\PP^s$.

We first claim that $X$ is lci. Indeed, locally on $X$, we can choose generators $f_1,\dots,f_c$ in~$\calI$, generating $\calI/\calI^2$ as a free module, such that $f_1,\dots,f_{c-1}$ give a basis for $\calE$, and $f_c$ maps to a basis vector for $\calQ$. But then the ideal $\calJ$ is generated by $f_1,\dots,f_{c-1}, f_c^2$; this is a regular sequence, so $X$ is lci.

Next, we show that $\calQ^{\otimes 2}$ occurs as a subbundle of $(\calJ/\calJ^2)|_{\PP^s}$. For this, note that
\[
(\calJ/\calJ^2)|_{\PP^s} = \calJ/\calI\calJ
\]
since $\calJ \subset \calI$. The multiplication map $\calI^{\otimes 2} \to \calI^2 \subset \calJ$ then gives a map
\[
\calQ^{\otimes 2} = (\calI/\calJ)^{\otimes 2} \to \calJ/\calI\calJ.
\]
On the other hand, $\calJ/\calI\calJ$ is an extension of $\calJ/\calI^2 \simeq \calE$ by $\calI^2/\calI\calJ$; the reader may then check that the previous map gives an isomorphism $\calQ^{\otimes 2} \simeq \calI^2/\calI\calJ$, and hence the desired subbundle $\calQ^{\otimes 2} \subset \calJ/\calI\calJ$.

Finally, we show that $\calN$ is $s$-ample, but not $(s-1)$-ample. Since $\dim X = s$, the $s$-ampleness is clear. For the rest, assume towards contradiction that $\calN$ is $(s-1)$-ample. Then, by Lemma~\ref{lemma:quotient:partially:ample}, any quotient of $\calN|_{\PP^s}$ is $(s-1)$-ample. On the other hand, our construction shows that the negative degree line bundle $\calQ^{\otimes -2}$ occurs as a quotient of~$\calN|_{\PP^s}$. Thus, $\calQ^{\otimes -2}$ must be $(s-1)$-ample, which is clearly absurd: $H^s(\PP^s, \Sym^t(\calQ^{\otimes -2}))$ is nonzero for all $t \gg 0$.

We make the above construction explicit in the case $s=1$ and $N=4$: Consider the polynomial ring $R=\FF[x,y,u,v,w]$ with the linear subspace $\PP^1$ defined by $I=(u,v,w)$ and its infinitesimal thickening $X:=\Proj R/J$ defined by $J=I^2+(uy-vx, vy-wx)$. It is easily seen that $X$ is lci. The rank of $H^0(X_t,\calO_{X_t})$ increases with $t$, so the maps
\[
H^0(X_{t+1},\calO_{X_{t+1}}) \to H^0(X_t,\calO_{X_t})
\]
cannot be isomorphisms for $t\gg0$. Nonetheless, as $X$ is lci, Corollary~\ref{corollary:finite:rank} guarantees that~$\ilim_t H^0(X_t,\calO_{X_t})$ has finite rank; indeed, it is rank one.
\end{example}

Next, we construct a class of examples for each $0 \le s < \dim X$:

\begin{example}
\label{ex:NormalNotAmpleReduced}
Let $Z$ be a normal projective variety of positive dimension, with an isolated singularity at $z \in Z$, such that $\underline{\Ext}^1(\Omega_Z^1, \calO_Z)$ is non-trivial and killed by the maximal ideal $\frakm_z \subset \calO_X$; for example, we may take $Z$ to be the projective cone over the smooth quadric surface in $\PP^3$. Let $Y$ be a smooth projective variety of dimension $s$. Let $X = Z \times Y$, so $X$ is lci and $\Sing(X) = \{z\} \times Y$ has dimension $s$. Choose a projective embedding $X \into \PP^n$, and let $\calN$ be the normal bundle of $X$ in this embedding. Then $\calN$ is $s$-ample by Theorem~\ref{theorem:normal:lci}. We claim that one cannot do better, i.e., $\calN$ is not $(s-1)$-ample. Assume towards contradiction that $\calN$ is $(s-1)$-ample. Then the same is true for $\calN|_{\Sing(X)}$.

On the other hand, the dual of the transitivity triangle for $X \into \PP^n$ gives, as in the proof of Theorem~\ref{theorem:normal:lci}, an exact triangle
\[
\calT_{\PP^n}|_X \to \calN \to \underline{\R\Hom}(\Omega^1_X, \calO_X)[1].
\]
In particular, on $\calH^0$, we obtain a quotient map
\[
\calN \onto \underline{\Ext}^1(\Omega^1_X, \calO_X) =: \calF.\]
On the other hand, since $X$ is a product and $Y$ is smooth, it is easy to see that
\[
\calF = \underline{\Ext}^1(\Omega^1_X,\calO_X) \simeq p_1^* \underline{\Ext}^1(\Omega^1_Z,\calO_Z).\]
By our assumption on $Z$, it follows that $\calF \simeq \calO_{\Sing(X)}^{\oplus r}$ is a trivial vector bundle supported on $\Sing(X)$. Thus, we can find a surjective map
\[
\calN|_{\Sing(X)} \onto \calO_{\Sing(X)}.
\]
Since $\calN|_{\Sing(X)}$ is assumed to be $(s-1)$-ample, the same is then true for $\calO_{\Sing(X)}$ by Lemma~\ref{lemma:quotient:partially:ample}. Since $\Sym^t(\calO_{\Sing(X)}) \simeq \calO_{\Sing(X)}$, it follows that for each integer $j$, we have
\[
H^s(\Sing(X), \calO_{\Sing(X)}(j)) = 0.
\]
However, this is clearly absurd: since $s$ equals the dimension of $\Sing(X)$, this group is nonzero for $j \ll 0$ by Serre duality.

We record some computations in the case $s=1$. Take $R$ to be a polynomial ring in $8$ variables, and $I$ an ideal such that $R/I$ is isomorphic to the Segre product
\[
\FF[x_0,x_1,x_2,x_3]/(x_1^2-x_2x_3)\ \# \ \FF[y_0,y_1].
\]
The ideal $I$ has height $4$, and the singular locus of $X:=\Proj R/I$ has dimension $1$, so Corollary~\ref{corollary:ext} implies that the maps ${\Ext^k_R(R/I^t,R)}_j \to {H^k_I(R)}_j$ must be injective for $k>5$ and $t\gg0$. However, the injectivity fails when $k=5$. The table records the numbers
\[
\rank{\Ext^5_R(R/I^t,R)}_j = \rank H^2(X_t,\calO_{X_t}(-8-j)).
\]
A singular cohomology calculation shows that $H^5_I(R)=H^8_\frakm(R)$, which gives the last row.

\begin{table}[H]
\begin{center}
\begin{tabular}{|c||*{9}{c|}}
\hline
\backslashbox{$t$}{$j$} & $-8$ & $-9$ & $-10$ & $-11$ & $-12$ & $-13$ & $-14$ & $-15$ & $-16$ \\
\hline
$1$ & \makebox[1.5em] & \makebox[1.5em] & \makebox[1.5em] & \makebox[1.5em] & \makebox[1.5em] & \makebox[1.5em] & \makebox[1.5em] & \makebox[1.5em] & \\
\hline
$2$ & $2$ & & & & & & & & \\
\hline
$3$ & $2$ & $16$ & $9$ & & & & & & \\
\hline
$4$ & $2$ & $16$ & $72$ & $64$ & $15$ & & & & \\
\hline
$5$ & $2$ & $16$ & $72$ & $240$ & $260$ & $120$ & $21$ & & \\
\hline
$6$ & $2$ & $16$ & $72$ & $240$ & $660$ & $792$ & $498$ & $168$ & $27$ \\
\hline
$7$ & $2$ & $16$ & $72$ & $240$ & $660$ & $1584$ & $2011$ & $1512$ & $735$\\
\hline
\hline
$\dlim_t$ & $1$ & $8$ & $36$ & $120$ & $330$ & $792$ & $1716$ & $3432$ & $6435$ \\[4pt]
\hline
\end{tabular}
\end{center}
\end{table}
\end{example}

As a corollary, we see that the bound $\codim(\Sing X)$ occurring in Theorem~\ref{theorem:constant} is sharp:

\begin{example}
\label{example:slocus:sharp}
Fix integers $0 \le s \le d$. Over a field of characteristic zero, choose an lci subscheme $X\subset\PP^n$ of dimension $d$ as in Example~\ref{ex:NormalNotAmpleReduced} if $s < d$, and as in Example~\ref{example:normal:not:ample} if~$s = d$. The normal bundle $\calN$ is $s$-ample but not $(s-1)$-ample by construction. Thus, by Lemma~\ref{lemma:ample:criterion}, we may choose an integer $j$ such that
\[
H^s(X, \Sym^t(\calN)(-j) \otimes K_X) \neq 0
\]
for infinitely many integers $t \gg 0$. We claim that the projective system $\{H^{d-s}(X_t,\calO_{X_t}(j))\}$ does not stabilize, i.e., the conclusion of Theorem~\ref{theorem:constant} fails for~$k = d-s$. To see this, fix some $t_0 \gg 0$ such that for all $t \ge t_0$, the maps
\[
H^{d-s-1}(X_{t+1}, \calO_{X_{t+1}}(j)) \to H^{d-s-1}(X_t, \calO_{X_t}(j))
\]
are isomorphisms; the existence of such a $t_0$ is ensured by Theorem~\ref{theorem:constant}. This ensures injectivity on the left the following standard exact sequence
\[
0 \to H^{d-s}(X, \Sym^t(\calN^\vee)(j)) \to H^{d-s}(X_{t+1}, \calO_{X_{t+1}}(j)) \stackrel{a_t}{\to} H^{d-s}(X_t, \calO_{X_t}(j)).
\]
Now the term on the left is dual to $H^s(X, \Sym^t(\calN)(-j) \otimes K_X)$ by Serre duality (and the identification of divided and symmetric powers in characteristic zero). By construction, this term is nonzero for infinitely many $t \gg t_0$. Thus, the map $a_t$ is not an isomorphism for infinitely many $t \gg t_0$, as wanted.
\end{example}

%%%%%%%%%%%%%%%%%%%%%%%%%%%%%%%%%%%%%%%%%%%%%%%%%%%%%%%%%%%%%%%%%%%%%%%%

\end{document}